\documentclass[final]{siamart190516}
\usepackage{amsmath,amssymb,stmaryrd}
\usepackage{amsfonts}
\usepackage{subcaption}
\usepackage{color}
\usepackage[nocompress]{cite}
\usepackage{hyperref}
\usepackage{graphicx}
\usepackage{array,multirow,graphicx}
\usepackage{float}
\newcommand*{\medcup}{\mathbin{\scalebox{1.5}{\ensuremath{\cup}}}}
\usepackage{makecell}
\usepackage{hhline}
\usepackage{cellspace}
\usepackage{mathtools}
\DeclarePairedDelimiterX\set[1]\{\}{\nonscript\,#1\nonscript\,}
\newcommand{\Omegahat}{\boldsymbol{\Omega}}
\renewcommand{\vec}[1]{\boldsymbol{\mathrm{#1}}}
\newcommand{\x}{\boldsymbol{\mathrm{x}}}
\newcommand{\nhat}{\boldsymbol{\mathrm{n}}}

\newsiamremark{remark}{Remark}
\newsiamremark{example}{Example}

\title{A note on $2\times 2$ block-diagonal preconditioning
  \thanks{\funding{BS is supported by the Department of Energy, National Nuclear Security Administration,
  under Award Number(s) DE-NA0002376. SO is supported by the Department of Energy Computational Science Graduate Fellowship under grant number DE-SC0019323.}}}

\author{Ben S. Southworth\thanks{Department of Applied Mathematics,
    University of Colorado, Boulder (\url{ben.southworth@colorado.edu}),
    \url{http://orcid.org/0000-0002-0283-4928}}
    \and
    Samuel S. Olivier\thanks{Department of Applied Science \& Technology, University of California, Berkeley
    (\url{solivier@berkeley.edu})}
}

\headers{A note on $2\times 2$ block-diagonal preconditioning}
{B. S. Southworth and S. S. Olivier}
\begin{document}
\allowdisplaybreaks

\maketitle
\begin{abstract}
For $2\times 2$ block matrices, it is well-known that block-triangular or block-LDU preconditioners
with an exact Schur complement (inverse) converge in at most two iterations for fixed-point
or minimal-residual methods. Similarly, for saddle-point matrices with a zero (2,2)-block,
block-diagonal preconditioners converge in at most three iterations for minimal-residual methods,
although they may diverge for fixed-point iterations. But, what happens for non-saddle-point
matrices and block-diagonal preconditioners with an exact Schur complement? This note proves
that minimal-residual methods applied to general $2\times 2$ block matrices, preconditioned with a block-diagonal
preconditioner, including an exact Schur complement, do not (necessarily) converge in a fixed
number of iterations. Furthermore, examples are constructed where (i) block-diagonal preconditioning
with an exact Schur complement converges no faster than block-diagonal preconditioning
using diagonal blocks of the matrix, and (ii) block-diagonal preconditioning with an approximate
Schur complement converges as fast as the corresponding block-triangular preconditioning.
The paper concludes by discussing some practical applications in neutral-particle transport, introducing
one algorithm where block-triangular or block-LDU preconditioning are superior to block-diagonal,
and a second algorithm where block-diagonal preconditioning is superior both in speed and simplicity. 
\end{abstract}

\section{$2\times 2$ block preconditioners}\label{sec:2x2}

Consider a block $2\times 2$ matrix equation,
\begin{align}\label{eq:block}
\begin{bmatrix} A_{11} & A_{12} \\ A_{21} & A_{22} \end{bmatrix}
	\begin{bmatrix} \mathbf{x}_1 \\ \mathbf{x}_2 \end{bmatrix}
	\begin{bmatrix} \mathbf{b}_1 \\ \mathbf{b}_2 \end{bmatrix},
\end{align}
where $A_{11}$ and $A_{22}$ are square (although potentially different sizes), and at least
one is invertible. Here we assume without loss of generality that $A_{11}$ is nonsingular,
and focus on the case of $A_{22} \neq \mathbf{0}$ (we refer to the case of $A_{22} = \mathbf{0}$
as saddle-point in this paper, unless otherwise specified).

Such systems arise in numerous applications, and are often solved iteratively using either
fixed-point or Krylov methods with some form of block preconditioning \cite{Benzi:2005kh,Wathen:2015hi}.
The four most common block preconditioners are block diagonal, block upper triangular, block lower triangular, 
and block LDU, which we will denote $D$, $U$, $L$, and $M$, respectively. It is generally believed that one of
the diagonal blocks in these preconditioners must approximate the appropriate Schur complement.
Given we assumed $A_{11}$ to be nonsingular, here we focus on the $(2,2)$-Schur complement,
${S}_{22} := A_{22} - A_{21}A_{11}^{-1}A_{12}$. Assume the action of ${S}_{22}^{-1}$
and $A_{11}^{-1}$ are available, and define block preconditioners as follows:
\begin{align} \label{eq:prec}
L &:= \begin{bmatrix} {A}_{11} & \mathbf{0} \\ A_{21} &{S}_{22}\end{bmatrix}, \hspace{6.4ex}
	U := \begin{bmatrix} {A}_{11} &A_{12} \\ \mathbf{0} & {S}_{22}\end{bmatrix}, \nonumber\\
D_{\pm} &:= \begin{bmatrix} {A}_{11} & \mathbf{0} \\ \mathbf{0} & \pm{S}_{22}\end{bmatrix}, \hspace{3ex}
	\widehat{D}_{\pm} := \begin{bmatrix} {A}_{11} & \mathbf{0} \\ \mathbf{0} & \pm{A}_{22}\end{bmatrix}, \\
M &:=  \begin{bmatrix} I & \mathbf{0} \\ A_{21}A_{11}^{-1} & I \end{bmatrix}  \begin{bmatrix} A_{11} & \mathbf{0} \\
	\mathbf{0} & {S}_{22}\end{bmatrix}\begin{bmatrix} I & A_{11}^{-1}A_{12} \\ \mathbf{0} & I \end{bmatrix}.\nonumber
\end{align}
Note, $\widehat{D}$ is a (more practical) variation on block-diagonal preconditioning that assumes $A_{22}$ is
nonsingular and its inverse available, rather than the Schur complement ${S}_{22}$. The
$\pm$ subscript on block-diagonal preconditioners corresponds to the sign on the (2,2)-block. For
the common case of a symmetric indefinite matrix $A$, with $A_{22}$ symmetric negative definite (SND)
and $A_{11}$ symmetric positive definite (SPD), block-diagonal preconditioners with a negative
(2,2)-block (like $D_-^{-1}$ or $\widehat{D}_-^{-1}$) are SPD, and a three-term recursion such as
preconditioned MINRES can be applied \cite{Paige:1975jv}.

In the following, we restate and tighten known results regarding fixed-point and minimal residual
iterations \cite{Paige:1975jv,Saad:1986hx} in exact arithmetic, using preconditioners in \eqref{eq:prec}, for
the general case \eqref{eq:block} and the saddle-point case, where $A_{22} = \mathbf{0}$. First, note the
following proposition. 

\begin{proposition}\label{prop:s22}
Let $A_{11}$ be invertible. Then dim(ker($A)) = $ dim(ker($S_{22}$)).
\end{proposition}
\begin{proof}
Note that if $A$ is singular, $\exists$ $\mathbf{s}$ such that $A\mathbf{s} = \mathbf{0}$. Under the
assumption that $A_{11}$ is nonsingular, expanding $A\mathbf{s}$ in $2\times 2$ block form
\eqref{eq:block} shows that $A$ can be singular if and only if ${S}_{22}\mathbf{s}_2 = \mathbf{0}$,
with corresponding ker$(A) = [-A_{11}^{-1}A_{12}\mathbf{s}_2; \mathbf{s}_2]$. Thus the dimension of
ker$(A)$ equals the dimension of ker$(S_{22})$, and $A$ is nonsingular if and only if $S_{22}$ is 
nonsingular.
\end{proof}

For block-diagonal preconditioner $D_{\pm}$, there is an implicit assumption that $A_{11}$ and
$S_{22}$ are invertible, in which case (by \Cref{prop:s22}) $A$ is invertible and $D_\pm^{-1}A$
is invertible. Proofs regarding the number of preconditioned minimal-residual iterations to exact
convergence with preconditioner $D_{\pm}$ in \cite{Murphy:2000hja,Ipsen:2001ui} include the
possibility of a zero eigenvalue in the characteristic polynomial. Since this cannot be the case,
the maximum number of iterations is three, not four as originally stated in \cite{Murphy:2000hja,Ipsen:2001ui}
(see \Cref{prop:diag}). 

\begin{proposition}[$2\times 2$ block preconditioners]
\text{ }
\begin{enumerate}
\itemsep0em
\item Fixed-point iterations \cite{19block} and minimal residual methods \cite{Murphy:2000hja,Ipsen:2001ui}
preconditioned with $L$, $U$, and $M$ converge in at most two iterations.

\item If $A_{11}$ and $A_{22}$ are nonsingular, convergence of fixed-point and minimal residual
methods preconditioned with $\widehat{D}_{\pm}$ is defined by convergence of an equivalent method
applied to the preconditioned Schur complements, $\pm A_{11}^{-1}{S}_{11}$ and $\pm A_{22}^{-1}{S}_{22}$
(roughly twice the number of iterations as the larger of the two) \cite{19block}.

\end{enumerate}
\end{proposition}

\begin{proposition}[$2\times 2$ block saddle-point preconditioners ($A_{22} = \mathbf{0})$]\label{prop:diag}
\text{ }
\begin{enumerate}
\itemsep0em
\item Fixed-point iterations preconditioned with $D_{\pm}$ do not necessarily converge \cite{19block}.

\item Minimal residual methods preconditioned with $D_{\pm}$ converge in at most three iterations
(see \cite{Murphy:2000hja,Ipsen:2001ui}, and \Cref{prop:s22}).
\end{enumerate}
\end{proposition}

Note that block-triangular and LDU-type preconditioners have a certain optimality -- in exact arithmetic
and using the exact Schur complement, convergence is guaranteed in two iterations, while convergence
using an approximate Schur complement is exactly defined by the preconditioned Schur complement-problem.
This provides general guidance in the development of block preconditioners, that the Schur
complement should be well approximated for a robust method.

The above propositions suggest one might also want a block-diagonal preconditioner to
approximate the Schur complement in one of the blocks. Indeed, block-diagonal preconditioners
are often designed to approximate a Schur complement in one of the blocks (for example, see
\cite{Siefert:2006fv}). However, it turns out that block-diagonal preconditioners with an exact Schur
complement, for $A_{22}\neq\mathbf{0}$, are \textit{not} optimal in the sense that block-triangular
and LDU preconditioners are. That is, block-diagonally preconditioned minimal-residual methods
are {not} guaranteed to converge in $\mathcal{O}(1)$ iterations (see \Cref{sec:diag} and
\Cref{prop:eig}), and convergence is not defined by the preconditioned Schur complement. 

Numerous papers have looked at eigenvalue analyses for block preconditioning (for example,
\cite{Cao:2003js,deSturler:2005ca,Bai:2009gq,Cao:2009gr,Murphy:2000hja,Ipsen:2001ui,
Bai:2005go,Siefert:2006fv,Notay:2014kv,Pestana:2014jj}). This paper falls into the series
of notes on the spectrum of various block preconditioners with exact inverses
\cite{Murphy:2000hja,Ipsen:2001ui,Cao:2003js,Cao:2009gr}, here presenting
results for the nonsymmetric and non-saddle-point, block-diagonal case. In addition to
the notes, this work
is perhaps most related to \cite{Siefert:2006fv}. There, a lengthy eigenvalue analysis is
done on approximate block-diagonal preconditioning, with approximations to $A_{11}^{-1}$
and $S_{22}^{-1}$, but the cases of exact inverses as in \eqref{eq:prec} are not directly
considered or discussed.

\Cref{sec:diag} presents the new theory, deriving the
spectrum of preconditioned operators $D_{\pm}^{-1}A$ and $\widehat{D}_{\pm}^{-1}A$
as a nonlinear function of the rank of $A_{12}$ and $A_{21}$, and the generalized
eigenvalues of matrix pencils $(A_{22},S_{22})$ or $(S_{22},A_{22})$. A brief
discussion on implications, particularly for symmetric matrices is given in \Cref{sec:diag:disc},
where examples are constructed to show that block-diagonal preconditioning with an exact
Schur compelment can converge relatively slowly, that block-diagonal preconditioning with an
approximate Schur complement can be as fast as, or significantly slower than, block triangular 
preconditioning (rather than $\approx2\times$ slower when $A_{22} = \mathbf{0}$ \cite{Fischer:1998vj}), and that block-diagonal 
preconditioning is likely more robust with symmetric indefinite matrices than SPD. 
A practical example in the simulation of neutral-particle transport is discussed in \Cref{sec:transport}.

\section{Block-diagonal preconditioning}\label{sec:diag}

It is shown in \cite{Murphy:2000hja,Ipsen:2001ui} that block-diagonal and block-triangular
preconditioners for saddle-point problems, with an exact Schur complement, have three and two
nonzero eigenvalues. Thus, the minimal polynomial for minimal residual methods
will be exact in a Krylov space of at most that size. 

\Cref{prop:eig} derives the full eigenvalue decomposition for general $2\times 2$ operators, 
preconditioned by the block-diagonal preconditioners $D_{\pm}^{-1}$ and $\widehat{D}_{\pm}^{-1}$. 
In particular, the spectrum of the preconditioned operator is fully defined by the generalized
eigenvalue problems $A_{22}\mathbf{y} = \widetilde{\lambda}S_{22}\mathbf{y}$
or $S_{22}\mathbf{y} = \widehat{\lambda}A_{22}\mathbf{y}$, depending on whether we
assume $S_{22}$ or $A_{22}$ are nonsingular, respectively. If $A_{22} = \mathbf{0}$, then
$\widetilde{\lambda} = 0$, and the results from \cite{Murphy:2000hja,Ipsen:2001ui}
immediately follow. For $A_{22}\neq\mathbf{0}$, the preconditioned operator can have more
eigenvalues, up to a full set of $n$ distinct eigenvalues (or, the generalized eigenvectors could
not form a complete basis for the space), and no such statements can be made
on the guaranteed convergence of minimal-residual methods. 

\begin{theorem}[Eigendecomposition of block-diagonal preconditioned operators]\label{prop:eig}
Assume that $A_{11}$ and $S_{22}$ are invertible. Let $\{\widetilde{\lambda},\mathbf{v}_2\}$ be
a generalized eigenpair of $A_{22}\mathbf{v}_2 = \widetilde{\lambda}S_{22}\mathbf{v}_2$. Then, the
spectra of the preconditioned operator for block-diagonal preconditioners $D_{\pm}^{-1}$ \eqref{eq:prec}
are given by:
\begin{align*}
\sigma(D_+^{-1}A) & = \left(\frac{\widetilde{\lambda} + 1}{2} \pm \frac{1}{2}\sqrt{(\widetilde{\lambda} - 1)(\widetilde{\lambda} + 3)}\right)
	\quad\medcup\quad \left\{ 1 \right\}, \\
\sigma({D}_-^{-1}A) & = \left( -\frac{ \widetilde{\lambda} - 1}{2} \pm \frac{1}{2}\sqrt{( \widetilde{\lambda} - 1)^2+4}\right)
	\quad\medcup\quad \left\{ \pm 1 \right\}.
\end{align*}
Assume that $A_{11}$ and $A_{22}$ are invertible. Now let $\{\widehat{\lambda},\mathbf{v}_2\}$ be
a generalized eigenpair of $S_{22}\mathbf{v}_2 = \widehat{\lambda}A_{22}\mathbf{v}_2$. Then, the
spectra of the preconditioned operator for block-diagonal preconditioners $\widehat{D}_{\pm}^{-1}$ \eqref{eq:prec}
are given by:
\begin{align*}
\sigma(\widehat{D}_+^{-1}A) & = \left(1 \pm \sqrt{1- \widehat{\lambda}}\right) \quad\medcup\quad \left\{ \pm 1\right\}, \\
\sigma(\widehat{D}_{-}^{-1}A) & = \left(\pm \sqrt{\widehat{\lambda}}\right) \quad\medcup\quad \left\{ \pm 1\right\}.
\end{align*}
The multiplicity of the $\{\pm1\}$ eigenvalues are given by the dimensions of the nullspace of
$A_{12}$ and $A_{21}$. Eigenvectors for each eigenvalue can also be written in a closed
form as a function of $\mathbf{v}_2$, and are derived in the proof (see \Cref{sec:diag:proof}).
\end{theorem}

Now assume $A_{12}A_{22}^{-1}A_{21}$ is nonsingular. By continuity of eigenvalues
as a function of matrix entries, 
\begin{align*}
\lim_{A_{22}\to 0} \sigma(D_+^{-1}A) & = \left\{ 1, \frac{1 \pm \mathrm{i}\sqrt{3}}{2}\right\}, \hspace{5ex}
\lim_{A_{22}\to 0} \sigma(D_-^{-1}A) = \left\{ 1, \frac{1 \pm \sqrt{5}}{2}\right\},
\end{align*}
and the maximum number of iterations for exact convergence of minimal-residual methods
converges to three. Note that the limit of $\sigma(D_-^{-1}A)$ does not include $-1$ because
by assuming $A_{12}A_{22}^{-1}A_{21}$ is nonsingular, there is an implicit assumption that 
ket$(A_{12}) = \emptyset$, which means $\not\exists$ an eigenvalue of $-1$ (see \Cref{lem:D-}).

\subsection{Discussion}\label{sec:diag:disc}

The primary implication of \Cref{prop:eig} is that block-diagonal preconditioned minimal-residual
methods with an exact Schur
complement do not necessarily converge in a fixed number of iterations. Recall block-diagonal
preconditioning is often applied to symmetric matrices for use with MINRES, and convergence
of such Krylov methods is better understood than GMRES, Here, we include
a brief discussion on the eigenvalues of block-diagonally preconditioned symmetric $2\times 2$ block
matrices with semi-definite diagonal blocks. 

Consider real-valued, symmetric, nonsingular matrices with nonzero saddle-point structure, that is,
\begin{align}\label{eq:sym}
A = \begin{bmatrix} A_{11} & A_{12} \\ A_{12}^T & A_{22} \end{bmatrix},
\end{align}
where $A_{11}$ is SPD and $A_{22}\neq\mathbf{0}$ is symmetric semidefinite
(positive or negative definite). Such matrices arise often in practice; see review papers
\cite{Benzi:2005kh,Wathen:2015hi} for many different examples.
From \Cref{prop:eig}, we can consider the spectrum of the preconditioned operator based on the
generalized spectrum of $(S_{22},A_{22})$. \Cref{fig:eig} plots the magnitude of eigenvalues
$\lambda\in\sigma(D_{\pm}^{-1}A)$ and $\lambda\in\sigma(\widehat{D}_{\pm}^{-1}A)$ as
a function of $\widetilde{\lambda}$ in a log-log scale (excluding zero on the logarithmic x-axis). 

\begin{figure}[!ht]
    \centering
    \begin{center}
        \begin{subfigure}[t]{0.49\textwidth}
            \includegraphics[width=\textwidth]{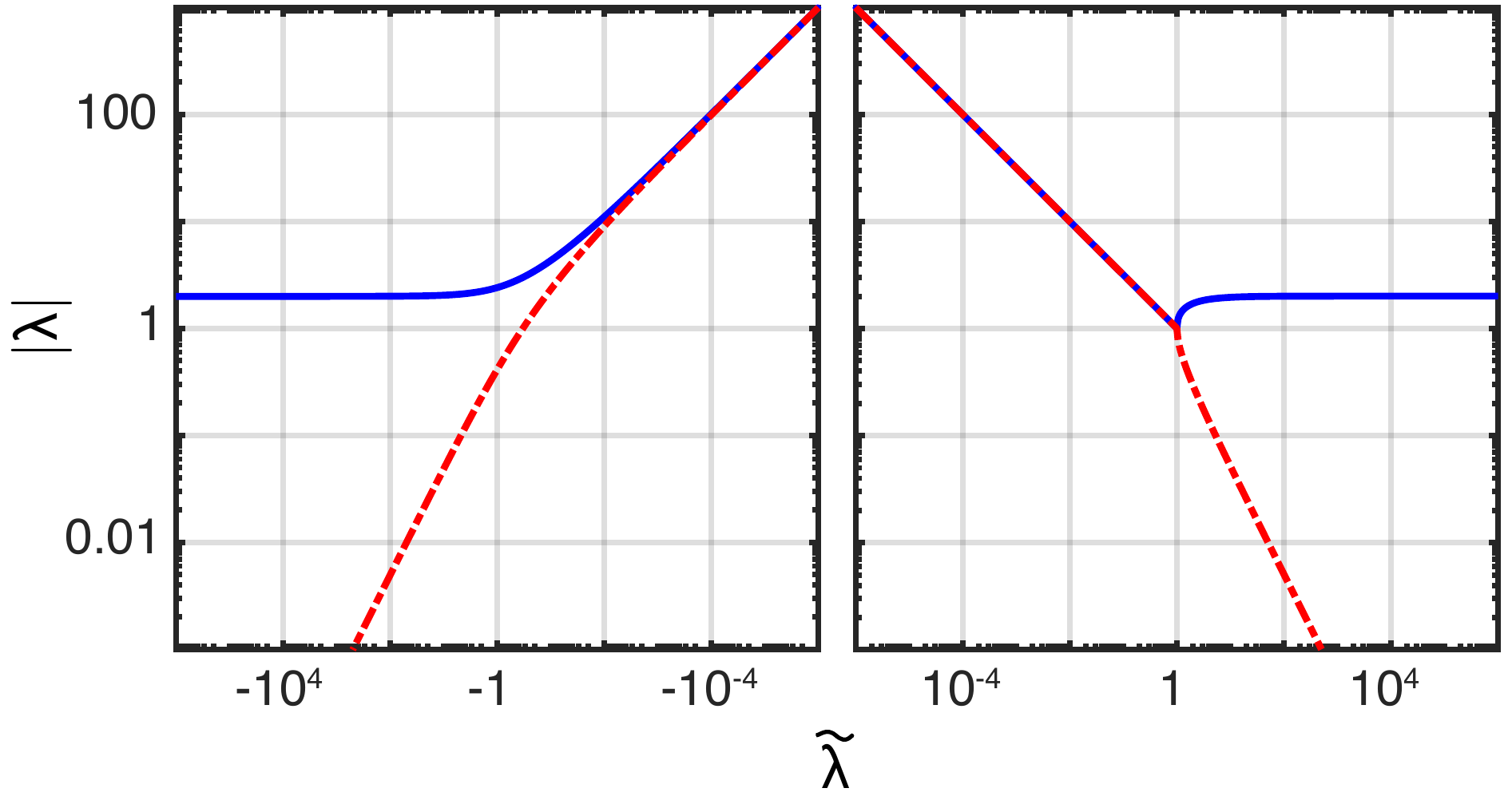}
	\caption{$\widehat{D}_+^{-1}$}
	\label{fig:Dp}
        \end{subfigure}
        \begin{subfigure}[t]{0.49\textwidth}
            \includegraphics[width=\textwidth]{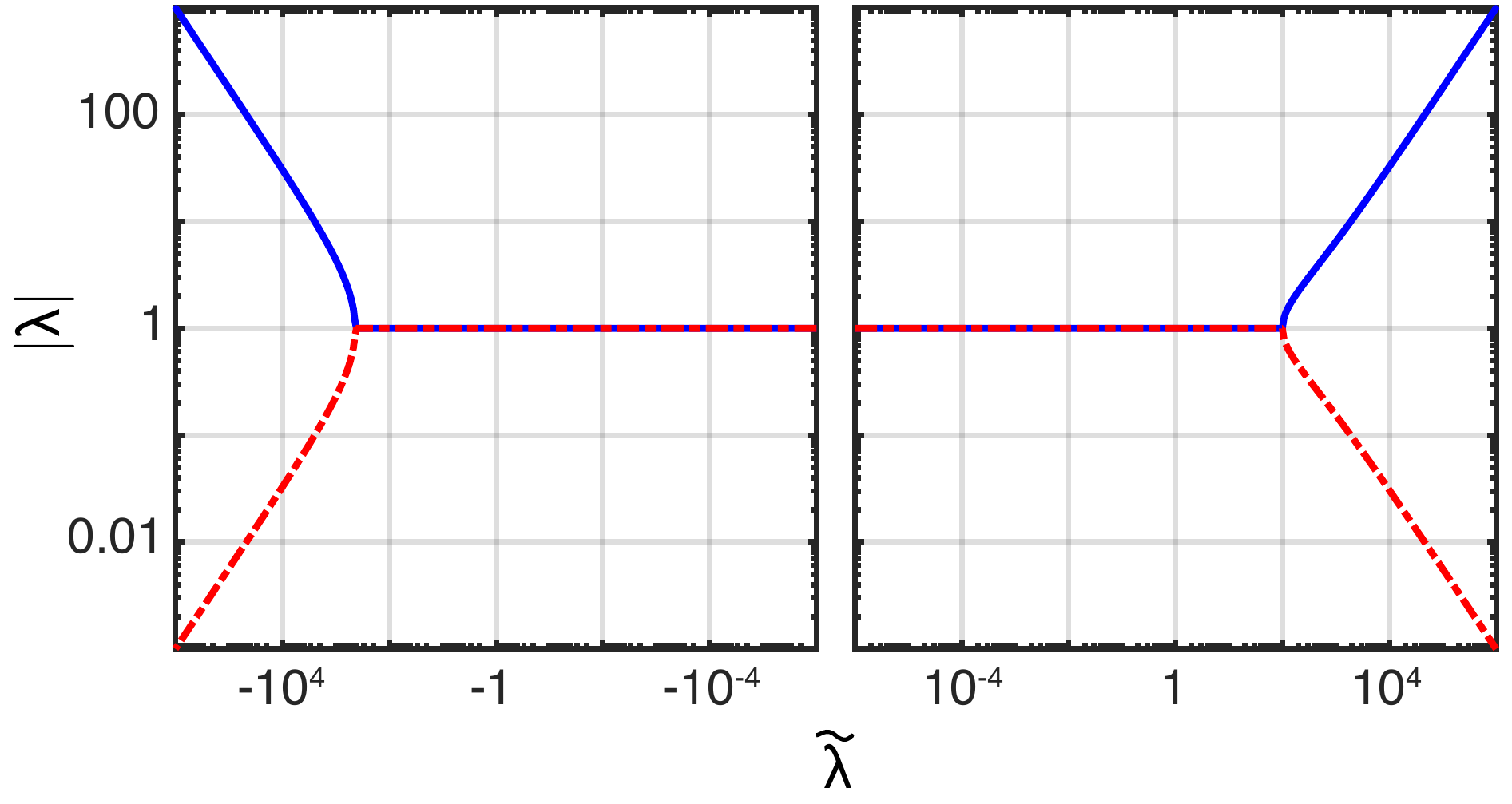}
	\caption{${D}_+^{-1}$}
        \label{fig:Sp}
        \end{subfigure}
        \\
        \begin{subfigure}[t]{0.49\textwidth}
            \includegraphics[width=\textwidth]{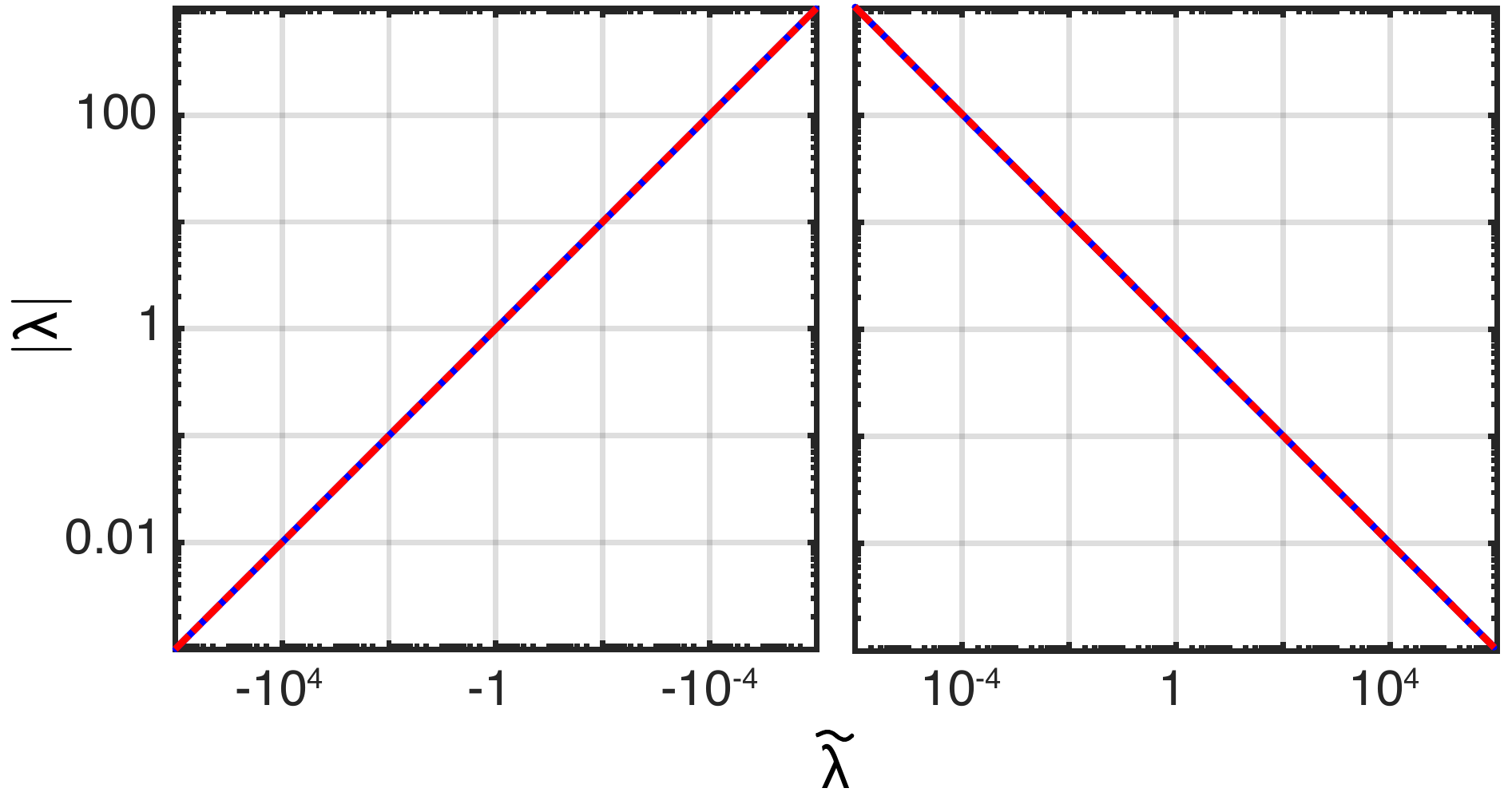}
	\caption{$\widehat{D}_-^{-1}$}
	\label{fig:Dn}
        \end{subfigure}
        \begin{subfigure}[t]{0.49\textwidth}
            \includegraphics[width=\textwidth]{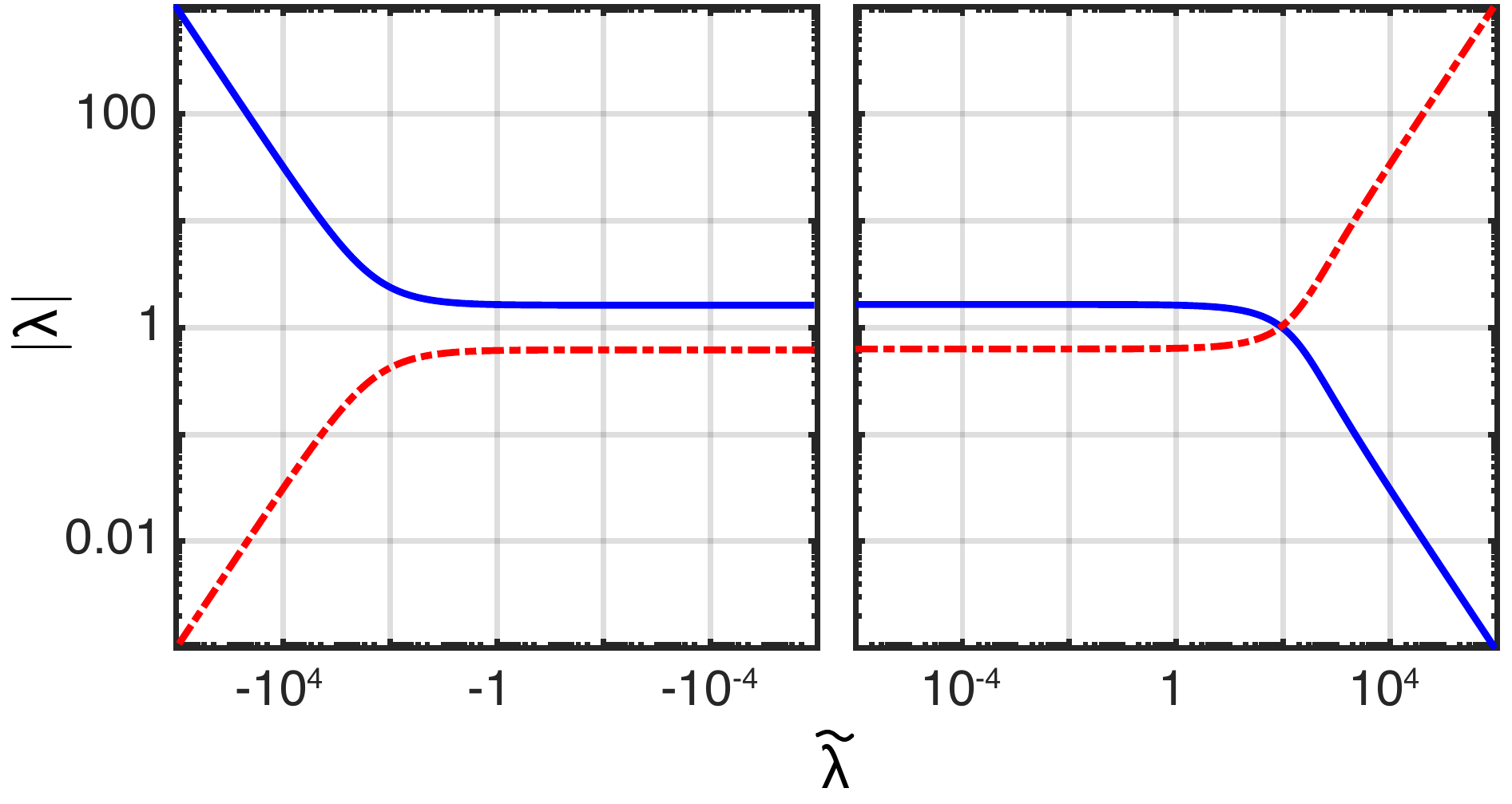}
	\caption{${D}_-^{-1}$}
        \label{fig:Sn}
        \end{subfigure}
    \end{center}
    \caption{Magnitude of eigenvalue of the preconditioned operator, $|\lambda|$, as a function of
    a real-valued generalized eigenvalue $\widetilde{\lambda}$ of $(A_{22}, S_{22})$,
    for block-diagonal preconditioners $D_{\pm}^{-1}$ and $\widehat{D}_{\pm}^{-1}$ (\Cref{prop:eig}).
    All formula for eigenvalues of $D_{\pm}^{-1}A$ and $\widehat{D}_{\pm}^{-1}A$ have a $\pm$ in
    the definition -- solid blue lines correspond to $+$ and dotted red lines correspond to $-$, for the given 
    preconditioner and eigenvalue $\widetilde{\lambda}$.
    }
    \label{fig:eig}
\end{figure}

Here we can see that when $|\widetilde{\lambda}| \lessapprox 100$, $D_{\pm}$ provides a
clearly superior preconditioning of eigenvalues than $\widehat{D}_{\pm}$. Conversely, for
$|\widetilde{\lambda}| \gtrapprox 1$, $\widehat{D}_{\pm}$ likely provides a comparable or 
potentially superior preconditioning of eigenvalues. The following example demonstrates
how SPD operators may be less amenable to block-diagonal Schur-complement 
preconditioning. 

\begin{example}[SPD operators]
Suppose $A$ is SPD. Then $A_{11},A_{22}$, and $S_{22}$ are also SPD, and 
$A_{12}^TA_{11}^{-1}A_{12}$ is symmetric positive semi-definite. Notice that 
$A_{22}^{-1}S_{22} = I - A_{22}^{-1}A_{12}^TA_{11}^{-1}A_{12}$, and the eigenvalues
of $A_{22}^{-1}A_{12}^TA_{11}^{-1}A_{12}$ are real and non-negative. It follows that
$\sigma(A_{22}^{-1}S_{22} )\subset (0,1]$ and $\widetilde{\lambda}\in \sigma(S_{22}^{-1}A_{22} )
\subset [1,\infty)$. Looking at \Cref{fig:eig}, eigenvalues of the preconditioned operator
are largely outside of the region where a Schur complement provides excellent
preconditioning of eigenvalues. 
\end{example}

Now let $A_{11}, A_{12}$, and $A_{22}$ be random $250\times 250$ matrices uniformly distributed
between $[-1,1]$. Then, for positive constants $\{c_o, c_1, c_2\}$, define
\begin{align}\label{eq:MN}
M = \begin{bmatrix} A_{11}^TA_{11}/c_1 & A_{12} / c_o \\ A_{12}^T /c_o & A_{22}^TA_{22}/c_2 \end{bmatrix}, 
\hspace{5ex}
N = \begin{bmatrix} A_{11}^TA_{11}/c_1 & A_{12} / c_o \\ A_{12}^T /c_o & \rho I -A_{22}^TA_{22}/c_2 \end{bmatrix}, 
\end{align}
where $\rho$ is chosen such that $M$ and $N$ have the same minimal magnitude eigenvalue, and
$(\rho I -A_{22}^TA_{22}/c_2) \leq 0$. Here $M$ and $N$ are symmetric $500\times 500$ matrices,
with SPD (1,1)-block, and SPD or SND (2,2)-block. \Cref{tab:ex} shows the number of block preconditioned
GMRES iterations to converge to $10^{-16}$ relative residual tolerance, using a one vector for the right-hand
side, and various $\{c_o,c_1,c_2\}$. Results are shown for block preconditioners $D_{\pm}$, $\widehat{D}_{\pm}$,
$L$, and $\widehat{L}$, a lower-triangular preconditioner using $A_{22}$ instead of $S_{22}$. 
{
\begin{table}[!htb]
\small
\centering
\begin{tabular}{|| c || c c |  c c | cc ||cc | cc | cc ||}\Xhline{1.25pt}
& \multicolumn{6}{c||}{$M$} & \multicolumn{6}{c||}{N} \\ \hline
$ c_o/c_1/c_2$ &$L$ & $\widehat{L}$ & $D_+$ & $\widehat{D}_+$ & $D_-$ & $\widehat{D}_-$ &
	$L$ & $\widehat{L}$ & $D_+$ & $\widehat{D}_+$ & $D_-$ & $\widehat{D}_-$  \\\hline
1/1/1 & 3 & 167 & 259 & 347 & ~276 & ~347 & 3 & 122 & 115 & 245 & 101 & 245 \\ 
20/1/1 & 3 & ~36 & ~74 & ~75 & ~~74  & ~~75 & 3 & 33 & ~61 & ~65 & ~60 & ~65 \\
1/10/1 & 3 & 434 & 399 & 879 & ~416 & ~880 & 3 & 343 & ~98 & 687 & ~88 & 688 \\
1/10/10 & 3 & 501 & 102 & 1001 & ~97 & 1001 & 3 & 496 & ~67 & 995 & ~62 & 993  \\
\hline\hline
\Xhline{1.25pt}
\end{tabular} 
\caption{Block-preconditioned GMRES iterations to $10^{-16}$ relative residual tolerance,
with right-hand side $\mathbf{b} = \mathbf{1}$.}
\label{tab:ex}
\end{table}
}

From \Cref{tab:ex}, note that block-diagonal preconditioning with
an exact Schur complement can make for a relatively poor preconditioner (see row 1/1/1). Moreover,
for a problem with some sense of block-diagonal dominance (in this case, scaling the off-diagonal blocks
by  1/20 compared with diagonal blocks; see row 20/1/1), the exact Schur complement provides almost
no improvement over a standard block-diagonal preconditioner, $\widehat{D}_{\pm}$.
Last, it is worth pointing out that for all cases, convergence of GMRES is faster 
applied to $N$ ($A_{22} \leq 0$) than $M$ ($A_{22}\geq 0$), in some cases up to $4\times$ faster,
which is generally consistent with \Cref{fig:eig} and the corresponding discussion.\footnote{Note that
$N$ and $M$ are still fundamentally different and there may be other factors that explain the
difference in convergence. However, for all tests with a given set of random matrices and fixed
$c_o/c_1/c_2$, the condition number, minimum magnitude eigenvalue, and maximum magnitude
eigenvalue of $M$ and $N$ are within 1\%.}

Finally, we consider how the theory extends to approximate block preconditioners.
Consider $N$ with constants $\rho=0$ and $c_o/c_1/c_2 = $ 1/10/1 \eqref{eq:MN}, and
two (artificial) approximate Schur complements,
\begin{align*}
S_{22}^* & = A_{22} - (1 - \varepsilon_*)A_{21}A_{11}^{-1}A_{12} , 
\hspace{5ex}S_{22}^{\times} = A_{22} - A_{21}A_{11}^{-1}A_{12} + \varepsilon_\times E,
\end{align*}
for error scaling $\varepsilon_*\in[0,1], \varepsilon_\times\geq 0$. Here, $E$ is a random error matrix with
entries uniformly distributed between $[-\eta,\eta]$, where $\eta$ is the average magnitude of entries in $A$.
Note that for $\varepsilon_*= \varepsilon_\times = 0$ we have $S_{22}^* = S_{22}^\times = S_{22}$,
while $S_{22}^*$ and $S_{22}^\times$ generally become decreasingly accurate approximations to
$S_{22}$ as $\varepsilon$ increases. Also, $\varepsilon_* = 1$ yields $S_{22}^* = A_{22}$. Then,
consider block preconditioners $D_*, L_*, D_\times$, and $L_\times$, which take the form of $D_+$
and $L$ in \eqref{eq:prec}, replacing $S_{22}\mapsto S_{22}^*$ or $S_{22}\mapsto S_{22}^\times$.
\Cref{fig:iters} shows the ratio of block-diagonal to block-lower-triangular preconditioned GMRES
iterations to $10^{-16}$ relative residual tolerance, for $S_{22}^*$ and $S_{22}^\times$, as a
function of the corresponding $\varepsilon$ error scaling. Results for $S_{22}^\times$ are also
shown with $A_{22} = \mathbf{0}$.

\begin{figure}[!hbt]
    \centering
    \includegraphics[width=0.6\textwidth]{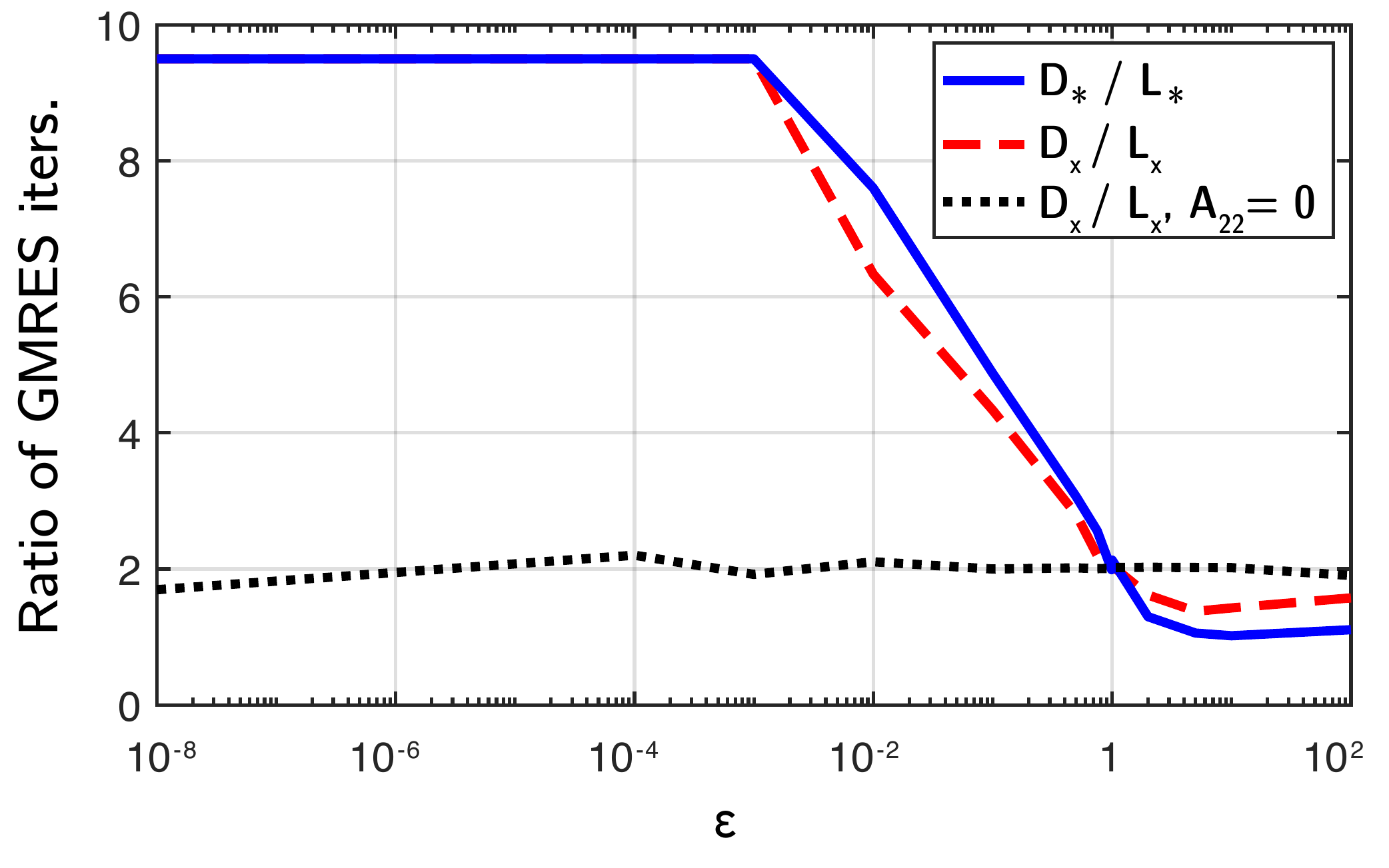}
    \caption{Ratio of block-diagonal to block lower-triangular preconditioned GMRES iterations to
    $10^{-16}$ relative residual, based on $S_{22}^*$ and $S_{22}^\times$, as a
    function of $\varepsilon_*$ or $\varepsilon_\times$. System matrix $N \in\mathbb{R}^{500\times 500}$
   is randomly generated as in \eqref{eq:MN} with $\rho = 0$.
    }
    \label{fig:iters}
\end{figure}

Note that for $\varepsilon \ll 1$ and $A_{22}\neq\mathbf{0}$, $L_*$ and $L_\times$ are significantly
more effective as preconditioners than $D_*$ and $D_\times$, in both cases requiring $\approx 10\times$
less GMRES iterations (with $L_*$ and $L_\times$ only requiring $2-4$ iterations). At $\varepsilon_* =1$,
$S_{22}^* = A_{22}$ and, consistent with theory in \cite{19block}, $D_*$ takes almost exactly twice as many
iterations as $L_*$ (interestingly, the same ratio holds for $D_\times/L_\times$, although this appears to
be coincidence). Finally, for $\varepsilon_* > 1$, $L_*$ takes almost the same number of iterations as
$D_*$. Despite both being relatively poor preconditioners, this again demonstrates that block-diagonal
preconditioning can yield convergence just as fast as block-triangular preconditioning in some cases.
Last, for $A_{22} = \mathbf{0}$, we observe the expected result that a block-diagonal preconditioner
will require twice as many iterations as block lower-triangular using the same Schur complement
approximation, when $A_{22} = \mathbf{0}$ \cite{Fischer:1998vj}.

The above discussion highlights some of the intricacies of Schur-complement preconditioning,
and some of the practical insights we can learn from this analysis. 
An interesting open question is whether the Schur complement is optimal in any sense for
block-diagonal preconditioners, or if there is a different ``optimal'' operator for the (2,2)-block. 
Of course, optimal in \textit{what} sense is perhaps the larger question, because as demonstrated
in \Cref{fig:eig} and \Cref{tab:ex} (and generally known \cite{Greenbaum:1996cp}), the eigenvalues
of the preconditioned operator can all be very nicely bounded, but GMRES can
observe arbitrarily slow convergence. 

\subsection{Proofs}\label{sec:diag:proof}

\begin{proof}[Proof of \Cref{prop:eig}]
The proof is presented as a sequence of lemmas for each preconditioner.
\begin{lemma}[Spectrum of $D_+^{-1}A$]\label{lem:D+}
\end{lemma}
\begin{proof}
By nature of the preconditioner, we assume that $A_{11}$ and $S_{22}$ are nonsingular and, thus, so is $A$. 
Noting the identity $A_{22} = S_{22} + A_{21}A_{11}^{-1}A_{12}$, one can expand $D_+^{-1}A$ to the block form
\begin{align*}
D_+^{-1}A & := \begin{bmatrix} A_{11}^{-1} & \mathbf{0} \\ \mathbf{0} & {S}_{22}^{-1} \end{bmatrix}
	\begin{bmatrix} A_{11}  & A_{12} \\ A_{21} & A_{22}\end{bmatrix}
= \begin{bmatrix} I \\ {S}_{22}^{-1}A_{21} \end{bmatrix} \begin{bmatrix} I & A_{11}^{-1}A_{12}\end{bmatrix} + 
	\begin{bmatrix} \mathbf{0} & \mathbf{0} \\ \mathbf{0} & I \end{bmatrix}.
\end{align*}
Now consider the corresponding eigenvalue problem,
\begin{align}\label{eq:eig0}
\left(\begin{bmatrix} I \\ {S}_{22}^{-1}A_{21} \end{bmatrix} \begin{bmatrix} I & A_{11}^{-1}A_{12}\end{bmatrix} + 
	\begin{bmatrix} \mathbf{0} & \mathbf{0} \\ \mathbf{0} & I \end{bmatrix}\right)
	\begin{bmatrix} \mathbf{x} \\ \mathbf{y} \end{bmatrix} & = 
	\lambda \begin{bmatrix} \mathbf{x} \\ \mathbf{y} \end{bmatrix},
\end{align}
which can be expanded as the set of equations
\begin{align*}
\mathbf{x} + A_{11}^{-1}A_{12}\mathbf{y} & = \lambda \mathbf{x}, \\
S_{22}^{-1}A_{21}(\mathbf{x} + A_{11}^{-1}A_{12}\mathbf{y}) & = (\lambda - 1)\mathbf{y}.
\end{align*}
If $A_{12}$ or $A_{21}$ are not full rank, then for each $\mathbf{y}_k\in$ ker$(A_{12})$ or
each $\mathbf{x}_k\in$ ker$(A_{21})$, there exist eigenpairs $\left\{ [ \mathbf{0}; \mathbf{y}_k] , 1\right\}$
or $\left\{ [\mathbf{x}_k; \mathbf{0}], 1\right\}$, respectively (the $;$ denotes a column vector).
This accounts for dim(ker($A_{21})) + $ dim(ker($A_{12}))$ eigenvalues. 
For $\mathbf{x}\not\in$ ker$(A_{21})$ and $\mathbf{y}\not\in$ ker$(A_{12})$,
solving for $\mathbf{x}$ yields $\mathbf{x} := \frac{-1}{1-\lambda}A_{11}^{-1}A_{12}\mathbf{y}$.
Plugging in reduces \eqref{eq:eig0} to an eigenvalue problem for $\mathbf{y}$,
\begin{align}\label{eq:eig1}
S_{22}^{-1}A_{21}A_{11}^{-1}A_{12} \mathbf{y} & = \frac{(1-\lambda)^2}{\lambda} \mathbf{y}.
\end{align}
Let $\widehat{\lambda} := (1- \lambda)^2/\lambda$ and assume $\lambda \neq 1$. Expanding, we have
$\lambda^2 - \lambda(2+\widehat{\lambda})+1 = 0$, and solving for $\lambda$ yields 
$\lambda = 1 + \frac{\widehat{\lambda}}{2} \pm \frac{1}{2}\sqrt{\widehat{\lambda}(\widehat{\lambda}+4)}$.

Now consider the eigenvalue problem in \eqref{eq:eig1} for $\widehat{\lambda}$. Applying $S_{22}$
to both sides and rearranging yields the generalized eigenvalue problem
$A_{22}\mathbf{y} = \frac{1+\widehat{\lambda}}{\widehat{\lambda}}A_{21}A_{11}^{-1}A_{12}\mathbf{y}$.
Subtracting $\frac{1+\widehat{\lambda}}{\widehat{\lambda}}A_{22}$ from both sides then yields the equivalent
generalized eigenvalue problem
\begin{align}
A_{22} \mathbf{y} & = (1+\widehat{\lambda}) S_{22}\mathbf{y}.\label{eq:eig3}
\end{align}
Thus, let $\widetilde{\lambda}$ be a generalized eigenvalue of $(A_{22}, S_{22})$ in \eqref{eq:eig3}.
Solving for $\widehat{\lambda}$ yields $\widehat{\lambda} := (\widetilde{\lambda} - 1)$ as a function of
generalized eigenvalues of $(A_{22}, S_{22})$ (note, for $\mathbf{y} \in$ ker$(A_{22})$, we
have $\widehat{\lambda} = -1$). Furthermore, the generalized eigenvectors in \eqref{eq:eig3} are
the same as eigenvectors in the (2,2)-block for the preconditioned operator \eqref{eq:eig0}.
Plugging $\widehat{\lambda}$ into the equation above for $\lambda$, we can express the
eigenvalues and eigenvectors of $D_+^{-1}A$ as a nonlinear function of the generalized
eigenvalues of $(A_{22}, S_{22})$,
\begin{align}\label{eq:DpA}
\sigma(D_+^{-1}A) = \left(\frac{\widetilde{\lambda} + 1}{2} \pm \frac{1}{2}\sqrt{(\widetilde{\lambda} - 1)(\widetilde{\lambda} + 3)}\right)
	\quad\medcup\quad \left\{ 1 \right\},
\end{align}
Eigenvectors of the former set in \eqref{eq:DpA} are given by $[\frac{-1}{1-\lambda}A_{11}^{-1}A_{12}\mathbf{y}; \mathbf{y}]$
for each generalized eigenpair $\{\mathbf{y}, \widehat{\lambda}\}$ of \eqref{eq:eig3} and eigenvalue 
$\lambda(\widehat{\lambda})$. Note that each generalized eigenpair of $(A_{22},S_{22})$ corresponds
to two eigenpairs of $D_+^{-1}A$, because $\lambda(\widehat{\lambda}) = 0$ is a quadratic equation in
$\lambda$. The latter eigenvalue $\lambda = 1$ in \eqref{eq:DpA} has multiplicity given by the sum of
dimensions of ker$(A_{12})$ and ker$(A_{21}$, with eigenvectors as defined previously in the proof. 
\end{proof}

\begin{lemma}[Spectrum of $D_{-}^{-1}A$]\label{lem:D-}
\end{lemma}
\begin{proof}
The proof proceed analogous to \Cref{lem:D+}. The block-preconditioned eigenvalue problem
can be expressed as the set of equations
\begin{align*}
\mathbf{x} + A_{11}^{-1}A_{12}\mathbf{y} & = \lambda \mathbf{x}, \\
-S_{22}^{-1}A_{21}(\mathbf{x} + A_{11}^{-1}A_{12}\mathbf{y}) & = (\lambda + 1)\mathbf{y}.
\end{align*}
If $A_{12}$ or $A_{21}$ are not full rank, then for each $\mathbf{y}_k\in$ ker$(A_{12})$ or
each $\mathbf{x}_k\in$ ker$(A_{21})$, there exist eigenpairs $\left\{ [ \mathbf{0}; \mathbf{y}_k] , -1\right\}$
or $\left\{ [\mathbf{x}_k; \mathbf{0}], 1\right\}$, respectively.
This accounts for dim(ker($A_{21})) + $ dim(ker($A_{12}))$ eigenvalues. 
For $\mathbf{x}\not\in$ ker$(A_{21})$ and $\mathbf{y}\not\in$ ker$(A_{12})$,
solving for $\mathbf{x}$ yields $\mathbf{x} := \frac{-1}{1-\lambda}A_{11}^{-1}A_{12}\mathbf{y}$.
Plugging in reduces to an eigenvalue problem for $\mathbf{y}$,
\begin{align}\label{eq:eign1}
S_{22}^{-1}A_{21}A_{11}^{-1}A_{12} \mathbf{y} & = \frac{(1-\lambda)(1+\lambda)}{\lambda} \mathbf{y}.
\end{align}
Let $\widehat{\lambda} := (1- \lambda)(1+\lambda)/\lambda$ and assume $\lambda \neq \pm1$. Expanding, we have
$\lambda^2 + \lambda\widehat{\lambda}- 1 = 0$, and solving for $\lambda$ yields 
$\lambda = -\frac{\widehat{\lambda}}{2} \pm \frac{1}{2}\sqrt{\widehat{\lambda}^2+4}$.
Plugging $\widehat{\lambda} := \widetilde{\lambda} - 1$ as in \Cref{lem:D+} yields the spectrum
of $D_-^{-1}A$ as a nonlinear function of the generalized eigenvalues of $(A_{22}, S_{22})$,
\begin{align} \label{eq:DmA}
\sigma({D}_-^{-1}A) = \left( -\frac{ \widetilde{\lambda} - 1}{2} \pm \frac{1}{2}\sqrt{( \widetilde{\lambda} - 1)^2+4}\right)
	\quad\medcup\quad \left\{ \pm 1 \right\}.
\end{align}
Eigenvectors of the former set in \eqref{eq:DmA} are given by $[\frac{-1}{1-\lambda}A_{11}^{-1}A_{12}\mathbf{y}; \mathbf{y}]$
for each generalized eigenpair $\{\mathbf{y}, \widehat{\lambda}\}$ of \eqref{eq:eig3} and eigenvalue 
$\lambda(\widehat{\lambda})$. Note that each generalized eigenpair of $(A_{22},S_{22})$ corresponds
to two eigenpairs of $D_+^{-1}A$, because $\lambda(\widehat{\lambda}) = 0$ is a quadratic equation in
$\lambda$. The latter eigenvalues $\lambda = -1$ and $\lambda = 1$ in \eqref{eq:DmA} have multiplicity
given by dimensions of ker$(A_{12})$ and ker$(A_{21})$, respectively, with eigenvectors as defined
previously in the proof. 
\end{proof}

\begin{lemma}[Spectrum of $\widehat{D}_+^{-1}A$]\label{lem:hD+}
\end{lemma}
\begin{proof}
Observe
\begin{align*}
\widehat{D}_+^{-1}A & = \begin{bmatrix} A_{11}^{-1} & \mathbf{0} \\ \mathbf{0} & A_{22}^{-1} \end{bmatrix}
	\begin{bmatrix} A_{11}  & A_{12} \\ A_{21} & A_{22}\end{bmatrix} 
= I + \begin{bmatrix} \mathbf{0} & A_{11}^{-1} A_{12} \\ A_{22}^{-1}A_{21} & \mathbf{0} \end{bmatrix}.\label{eq:22eig0}
\end{align*}
with eigenvalues and eigenvectors determined by the off-diagonal matrix on the right.
If $A_{12}$ or $A_{21}$ are not full rank, then for each $\mathbf{y}_k\in$ ker$(A_{12})$ or
each $\mathbf{x}_k\in$ ker$(A_{21})$, there exist eigenpairs $\left\{ [ \mathbf{0}; \mathbf{y}_k] , 1\right\}$
or $\left\{ [\mathbf{x}_k; \mathbf{0}], 1\right\}$, respectively. For $\mathbf{x}\not\in$ ker$(A_{21})$
and $\mathbf{y}\not\in$ ker$(A_{12})$, we consider eigenvalues of the off-diagonal matrix in \eqref{eq:22eig0}.
Squaring this matrix results in the block-diagonal eigenvalue problem 
\begin{align}
    \begin{bmatrix} A_{11}^{-1} A_{12}A_{22}^{-1}A_{21} & \mathbf{0} \\ \mathbf{0} & A_{22}^{-1}A_{21}A_{11}^{-1}A_{12}\end{bmatrix}
	\begin{bmatrix} \mathbf{x} \\ \mathbf{y} \end{bmatrix} & = 
	\widehat{\lambda} \begin{bmatrix} \mathbf{x} \\ \mathbf{y} \end{bmatrix} \label{eq:eigS}.
\end{align}
We can restrict ourselves to the case of $\mathbf{x}\not\in$ ker$(A_{21})$ and $\mathbf{y}\not\in$ ker$(A_{12})$.
Note by similarity that the spectrum of $A_{22}^{-1}A_{21}A_{11}^{-1}A_{12}$ is equivalent to that of
$A_{21}A_{11}^{-1}A_{12}A_{22}^{-1}$. Now suppose $A_{11}^{-1}A_{12}A_{22}^{-1}A_{21}\mathbf{x} =
\lambda\mathbf{x}$, and let $\mathbf{w} := A_{21}\mathbf{x}$. Then $A_{21}A_{11}^{-1}A_{12}A_{22}^{-1}\mathbf{w} 
= A_{21}(A_{11}^{-1}A_{12}A_{22}^{-1}A_{12}\mathbf{x}) = \lambda A_{21}\mathbf{x} = \lambda\mathbf{w}$.
Thus if $\{\lambda,\mathbf{x}\}$ is a nonzero eigenpair of $A_{11}^{-1} A_{12}A_{22}^{-1}A_{21}$, then
$\{\lambda, A_{22}^{-1}A_{21}\mathbf{x}\}$ is an eigenpair of $A_{22}^{-1}A_{21}A_{11}^{-1}A_{12}$.
To that end, eigenvalues of \eqref{eq:eigS} consist of 0, corresponding to the kernels of $A_{21}$ and $A_{12}$,
and the nonzero eigenvalues of the lower diagonal block (or upper), all with multiplicity two. 

Now note that $A_{22}^{-1}A_{21}A_{11}^{-1}A_{12} = I - A_{22}^{-1}S_{22}$, and let $\widetilde{\lambda}$ be a
generalized eigenvalue of $(S_{22},A_{22})$. Then, $\widehat{\lambda} = 1 - \widetilde{\lambda}$, and
the spectrum of $D_{-}^{-1}A$ can be written as a nonlinear function of the generalized eigenvalues of
$(S_{22}, A_{22})$,
\begin{align*}\label{eq:DhpA}
\sigma(\widehat{D}_+^{-1}A) & = \left(1 \pm \sqrt{1- \widetilde{\lambda}}\right) \quad\medcup\quad \left\{ \pm 1\right\},
\end{align*}
with inclusion of the latter set of eigenvalues depending on the rank of $A$ and $A_{12}$/$A_{21}$,
respectively. 

Note, each vector $\mathbf{v}_k\in$ ker$(A)$ correspond to a $\mathbf{y}_k\in$ ker$(S_{22})$. In the
context of the generalized eigenvalues of $(S_{22},A_{22})$, each $\widetilde{\lambda} = 0$ is one-to-one
with an eigenvalue $\lambda = 0$ of $\widehat{D}_+^{-1}A$, which can be seen by plugging
$\widetilde{\lambda} = 0$ into \eqref{eq:DhpA}.
\end{proof}

\begin{lemma}[Spectrum of $\widehat{D}_{-}^{-1}A$]\label{lem:hD-}
\end{lemma}
\begin{proof}
Observe
\begin{align}\label{eq:eig0hD-}
\widehat{D}_{-}^{-1}A & = \begin{bmatrix} A_{11}^{-1} & \mathbf{0} \\ \mathbf{0} & -A_{22}^{-1} \end{bmatrix}
	\begin{bmatrix} A_{11}  & A_{12} \\ A_{21} & A_{22}\end{bmatrix} 
= -I + \begin{bmatrix} 2I & A_{11}^{-1} A_{12} \\ -A_{22}^{-1}A_{21} & \mathbf{0} \end{bmatrix},
\end{align}
with eigenvalues and eigenvectors determined by the saddle-point matrix on the right. If $A_{12}$ or $A_{21}$
are not full rank, then for each $\mathbf{y}_k\in$ ker$(A_{12})$ or each $\mathbf{x}_k\in$ ker$(A_{21})$, there
exist eigenpairs $\left\{ [ \mathbf{0}; \mathbf{y}_k] , -1\right\}$ or $\left\{ [\mathbf{x}_k; \mathbf{0}], 1\right\}$,
respectively. For $\mathbf{x}\not\in$ ker$(A_{21})$
and $\mathbf{y}\not\in$ ker$(A_{12})$, analogous to \Cref{lem:D+} we can formulate an eigenvalue
problem for the saddle-point matrix, with eigenvalue $\widehat{\lambda}$ and block eigenvector
$[\mathbf{x}; \mathbf{y}]$.
Solving for $\mathbf{x} = \frac{1}{\widehat{\lambda}-2} A_{11}^{-1}A_{12}\mathbf{y}$ and plugging
into the equation for $\mathbf{y}$ yields an equivalent reduced eigenvalue problem,
\begin{align}\label{eq:eighD-}
A_{22}^{-1}A_{21}A_{11}^{-1}A_{12} \mathbf{y} & = \widehat{\lambda}(2 -  \widehat{\lambda})\mathbf{y}.
\end{align}
Note, this is well-posed in the sense that we only care about $\mathbf{y}\not\in$ ker$(A_{12})$
and $\mathbf{x} = \frac{1}{\widehat{\lambda}-2} A_{11}^{-1}A_{12}\mathbf{y} \not\in$ ker$(A_{21})$. 
Noting that $A_{22}^{-1}A_{21}A_{11}^{-1}A_{12} = I - A_{22}^{-1}S_{22}$, \eqref{eq:eighD-}
is equivalent to the generalized eigenvalue problem
\begin{align}\label{eq:genhD-}
S_{22}\mathbf{y} & = (1 - \widehat{\lambda})^2A_{22}\mathbf{y}.
\end{align}
Letting $\widetilde{\lambda}$ be a generalized eigenvalue of $(S_{22}, A_{22})$ in \eqref{eq:genhD-}
and solving for $\widehat{\lambda}$ yields $\widehat{\lambda} = 1\pm \sqrt{\widetilde{\lambda}}$. Including
the minus identity perturbation in \eqref{eq:eig0hD-} yields the spectrum of $D_{-}^{-1}A$ as a
nonlinear function of the generalized eigenvalues of $(S_{22}, A_{22})$ in \eqref{eq:genhD-},
given by
\begin{align*}
\sigma(\widehat{D}_{-}^{-1}A) = \left(\pm \sqrt{\widetilde{\lambda}}\right) \quad\medcup\quad \left\{ \pm 1\right\},
\end{align*}
Eigenvectors corresponding to eigenvalues $\sqrt{\widetilde{\lambda}}$ are given by 
$[\frac{-1}{1-\lambda}A_{11}^{-1}A_{12}\mathbf{y}; \mathbf{y}]$ for each generalized eigenpair
$\{\mathbf{y}, \widehat{\lambda}\}$ of \eqref{eq:eig3} and eigenvalue 
$\lambda(\widehat{\lambda})$. Note that each generalized eigenpair of $(S_{22},A_{22})$ corresponds
to two eigenpairs of $\widehat{D}_-^{-1}A$, because $\lambda(\widehat{\lambda}) = 0$ is a quadratic equation in
$\lambda$. The latter eigenvalue $\lambda = 1$ in \eqref{eq:DpA} has multiplicity given by the sum of
dimensions of ker$(A_{12})$ and ker$(A_{21}$, with eigenvectors as defined previously in the proof. 
\end{proof}

\end{proof}

\section{Applications in transport: $\mathcal{H}_1\otimes L_2$ VEF}\label{sec:transport}

\subsection{The variable Eddington factor (VEF) method}\label{sec:transport:vef}

The steady-state, mono-energetic, discrete-ordinates, linear Boltzmann equation with isotropic scattering
and source is given by
\begin{equation} \label{eq:transport}
	\begin{aligned}
		\Omegahat_d \cdot \nabla \psi_d(\x) + \sigma_t(\x) \psi_d(\x) &= \frac{\sigma_s(\x)}{4\pi} \sum_{d'=1}^N w_{d'} \psi_{d'}(\x) + \frac{Q(\x)}{4\pi} \,, \quad \x \in \mathcal{D} \,, \\
		\psi_d(\x) &= \overline{\psi}(\x,\Omegahat_d) \,, \quad \x \in \partial \mathcal{D} \ \mathrm{and} \ \Omegahat_d \cdot \nhat < 0 \,. 
	\end{aligned}
\end{equation}
Here, the direction of particle motion, $\Omegahat$, is discretized into $N$ angles corresponding
to a quadrature rule for the unit sphere, $\{\Omegahat_d, w_d\}_{d=1}^N$, for quadrature weights $w_d>0$.
A standard approach to solve \eqref{eq:transport} is to (independently) invert the
left-hand side for each $\psi_d$, update $\sum_{d'=1}^N w_{d'} \psi_{d'}(\x)$,
and repeat until convergence. Unfortunately, for many applications of interest, this process can converge arbitrarily slowly.

An alternative scheme for solving \eqref{eq:transport} is the Variable Eddington factor
(VEF) method, a nonlinear iterative scheme that builds an exact reduced-order model of the transport
equation \cite{goldin,mihalas}. The VEF reduced-order model is given by 
\begin{align}\label{eq:VEF}
\begin{split}
\nabla \cdot \vec{J}(\x) + \sigma_a(\x) \varphi(\x) & = Q(\x) \, \\
\nabla \cdot \left[ \vec{E}(\x) \varphi(\x) \right] + \sigma_t(\x) \vec{J}(\x) & = 0 \,.
\end{split}
\end{align}
It is derived by taking the zeroth and first angular moments of \eqref{eq:transport}, and introducing the Eddington tensor
as a closure:
\begin{equation}\label{eq:E}
		\vec{E}(\x) = \frac{\sum_{d=1}^N \Omegahat_d \otimes \Omegahat_d \, \psi_d(\x) w_d}{\sum_{d=1}^N \psi_d(\x) w_d} \,. 
\end{equation}
The algorithm is then to replace the scattering sum $\sum_{d'=1}^N w_{d'} \psi_{d'}(\x)$ in \eqref{eq:transport}
with the solution of the VEF equations, $\varphi(\x)$, invert \eqref{eq:transport} for each angle, update the 
Eddington tensor, and solve the VEF equations.

Here we consider an $\mathcal{H}^1(\mathcal{D})\otimes L^2(\mathcal{D})$ mixed finite
element discretization of \eqref{eq:VEF}, as in \cite{vef_paper,vef_proceedings}. This yields the $2\times 2$
block system
\begin{equation} \label{eq:vef_system}
		\begin{bmatrix} M_t & -G \\ B & M_a \end{bmatrix} \begin{bmatrix} \vec{J}\\ \vec{\varphi} \end{bmatrix} = \begin{bmatrix} \vec{g} \\ \vec{f} \end{bmatrix} \,,
\end{equation}
where $M_t$ is an $\mathcal{H}^1(\mathcal{D})$ mass matrix, $M_a$ a
(block diagonal) $L^2(\mathcal{D})$ mass matrix, $B$ the
$L^2(\mathcal{D})\otimes \mathcal{H}^1(\mathcal{D})$ divergence matrix, and $G$ the
$\mathcal{H}^1(\mathcal{D}) \otimes L^2(\mathcal{D})$ Eddington gradient matrix, defined as
$\vec{v}^T G \vec{\varphi} = \int_\mathcal{D} \nabla \vec{v} : \vec{E} \varphi \, d\x$.
Note that for $\sigma_a(x) = \sigma_t(x) - \sigma_s(x) = 0$, $M_a = \mathbf{0}$ and \eqref{eq:vef_system}
simplifies to a saddle-point system with zero (2,2)-block. Due to the Eddington tensor,
$B \neq -G^T$ and, thus, at every VEF iteration a non-symmetric $2\times 2$ block operator of the
form in \eqref{eq:block} must be solved.

\subsection{Solution of $\mathcal{H}_1\otimes L_2$ VEF discretization}

There are a number of advantages to VEF over other methods of solving \eqref{eq:transport} \cite{vef_paper,
vef_proceedings}, but the system in \eqref{eq:vef_system} remains difficult to solve.
However, because the diagonal blocks of \eqref{eq:vef_system}
are mass matrices and relatively easy to invert, direct application or preconditioning of the Schur complement,
$S_{22} := M_a + B M_t^{-1} G$, is feasible, which corresponds to a drift-diffusion-like equation.
The remaining challenge is that $M_t$ is an
$\mathcal{H}^1(\mathcal{D})$ mass matrix and, although the action of $M_t^{-1}$ can be
computed rapidly using Krylov methods, a closed form to construct $S_{22} := M_a + B M_t^{-1} G$ is
typically not available. A common solution is to approximate $M_t$ using quadrature lumping, resulting in
a diagonal approximation, $\widetilde{M}_t$. Thus, consider the four preconditioners,
\begin{align}\label{eq:vef_prec}
\begin{split}
D & = \begin{bmatrix} M_t & \mathbf{0} \\ \mathbf{0} & M_a + BM_t^{-1} G \end{bmatrix} , \hspace{5.75ex}
	L = \begin{bmatrix} M_t & \mathbf{0} \\ B & M_a + B M_t^{-1} G \end{bmatrix} , \\
\widetilde{D}&  = \begin{bmatrix} M_t & \mathbf{0} \\ \mathbf{0} & M_a + B \widetilde{M_t}^{-1} G \end{bmatrix} \,,\hspace{5ex}
		\widetilde{L} = \begin{bmatrix} M_t & \mathbf{0} \\ B & M_a + B \widetilde{M_t}^{-1} G \end{bmatrix} .
\end{split}
\end{align}

Consider a 1d discretization of \eqref{eq:VEF}, where $\mathrm{J}$
is discretized with quadratic $\mathcal{H}^1(\mathcal{D})$ and $\varphi$ with linear $L^2(\mathcal{D})$,
with $\sigma_t(x) = Q(x) = 1$ and absorption $\sigma_a\in\{0,0.9\}$ to demonstrate the effect of the
$A_{22}$ block ($A_{22} = \mathbf{0}$ when $\sigma_a = 0$).
In the limit of $\sigma_t \gg 1$, $E \to \frac{1}{3}I$, wherein scaling the second row of \eqref{eq:vef_system}
by $-1/3$ then results in a symmetric operator. Although this is unlikely to occur on the whole domain in practice,
we perform tests on this case as well to demonstrate that results do not depend on nonsymmetry in \eqref{eq:vef_system}. 
Table \ref{tab:data} shows the number of GMRES iterations to converge to $10^{-10}$ relative residual tolerance
using block preconditioners in \eqref{eq:vef_prec}. 
{\begin{table}[!htb]
\small
\centering

\begin{tabular}{||c | c | c || cc | cc ||}\Xhline{1.25pt}
\multicolumn{3}{||c||}{} & \multicolumn{2}{c|}{Unlumped} & \multicolumn{2}{c||}{Lumped} \\ \hline
& $\sigma_a$ & $N$ & Diagonal & Triangular & Diagonal & Triangular \\ \hline\hline
	\parbox[t]{2mm}{\multirow{4}{*}{\rotatebox[origin=c]{90}{Nonsymm.}}}
& 0 & 401 & 2 & 2 & 44 & 23 \\
& 0 & 4001 & 4 & 4 & 53 & 25 \\\cline{2-7}
& 0.9 & 401 & 12 & 2 & 57 & 27\\
& 0.9 & 4001 & 12 & 4 & 77 & 32 \\\hline\hline
	\parbox[t]{2mm}{\multirow{4}{*}{\rotatebox[origin=c]{90}{Symm.}}}
& 0 & 401 & 2 & 2 & 44 & 23 \\
& 0 & 4001 & 4 & 4 & 55 & 25 \\ \cline{2-7}
& 0.9 & 401 & 12 & 2 & 57 & 27 \\
& 0.9 & 4001 & 12 & 4 & 76 & 26 \\
\hline\hline
\Xhline{1.25pt}
\end{tabular} 
\caption{ }
\label{tab:data}
\end{table}
}
The unlumped block-diagonal and triangular preconditioners converge in two or four iterations when the
$A_{22}$ block is zero, largely consistent with theory in \cite{Murphy:2000hja,Ipsen:2001ui}.
However, when 
$M_a = A_{22} \neq \mathbf{0}$, block-diagonal preconditioning requires 12 iterations. Although this is
better than the several hundred seen in \Cref{tab:ex}, it is still $3-6\times$ slower than block lower
triangular. Moreover, the example problem used here is relatively simple and one-dimensional -- initial
experiments on harder problems have required a larger number of iterations. We can also see how
sensitive Schur-complement preconditioning can be to approximations, where the lumped preconditioners
increase iteration counts by $10\times$ or more. In practice, lumping $M_t$ is likely necessary
because we cannot directly form $M_a + BM_t^{-1}G$, which makes construction of effective
preconditioners difficult. Considering this problem in the context of Schur-complement preconditioning
has suggested reformulating \eqref{eq:vef_system} using a lumped $M_t$ (which we can solve in $\approx 2$
iterations), block-triangular preconditioners, and handling the difference between a lumped and non-lumped
$M_t$ in the larger nonlinear VEF iteration. This is ongoing work, and an interesting application of
Schur-complement preconditioning.

\section{Conclusions}\label{sec:conc}

The simple lesson from \Cref{sec:transport} is that implementing a block-triangular (or block-LDU if symmetry
is important) preconditioner may provide a significant speedup over block-diagonal preconditioners
($\ll1/2$ the iteration count), and should be considered in practice for $A_{22}\neq\mathbf{0}$ and Schur-complement
approximation $\widehat{S}_{22}\neq A_{22}$. The larger point of this paper
is that which type of block preconditioner to use is largely problem specific, and should be considered
on a case-by-case basis. When $A_{22} = \mathbf{0}$, it is generally known that block-diagonal preconditioners
will require roughly twice as many minimal-residual iterations to converge as block-triangular or block-LDU \cite{Fischer:1998vj},
and it is straightforward to estimate the associated computational costs for each preconditioner and pick the most efficient
choice. For $A_{22}\neq\mathbf{0}$, examples in \Cref{sec:diag:disc} demonstrate that block-diagonal and
block-triangular preconditioning can converge in a similar or very different number of iterations. Another
recently developed $2\times 2$ block preconditioner for transport problems found that the block-triangular
variation rarely showed any reduction in iteration count over block-diagonal, while computing the action of
the off-diagonal blocks made the triangular variation several times more expensive \cite{19hetdsa}. In that case,
implementing  the triangular variation is also non-trivial, and the block-diagonal preconditioner is clearly
superior in terms of performance and ease of implementation. 

\bibliographystyle{siamplain}
\bibliography{refs}
\end{document}